\documentclass[11pt]{article}
\usepackage{amsmath, amscd, amssymb, latexsym, epsfig, color, amsthm,enumerate, mathtools, tikz-cd,url}
\usepackage[all]{xy}
\usepackage{verbatim}
\usepackage{esvect}
\numberwithin{equation}{section}
\def\proof{\smallskip\noindent {\it Proof: \ }}
\def\endproof{\hfill$\square$\medskip}
\newtheorem{theorem}{Theorem}[section]
\newtheorem{proposition}[theorem]{Proposition}

\newtheorem{corollary}[theorem]{Corollary}
\newtheorem{conjecture}[theorem]{Conjecture}
\newtheorem{lemma}[theorem]{Lemma}

\theoremstyle{definition}

\newtheorem{example}[theorem]{Example}

\newtheorem{remark}[theorem]{Remark}

\DeclareMathOperator{\skel}{Skel}
\DeclareMathOperator\lk{\mathrm{lk}}

\DeclareMathOperator\st{\mathrm{st}}

\DeclareMathOperator{\soc}{\mathrm{Soc}}
\DeclareMathOperator{\tor}{\mathrm{Tor}}

\DeclareMathOperator{\conv}{\mathrm{conv}}

\DeclareMathOperator{\supp}{\mathrm{supp}}

\newcommand{\m}{{\mathbf m}}

\newcommand{\R}{{\mathbb R}}
\newcommand{\RR}{{\mathcal R}}
\newcommand{\Q}{{\mathbb Q}}
\newcommand{\Z}{{\mathbb Z}}

\newcommand{\Stress}{\mathcal S}

\title{Affine stresses, inverse systems, and reconstruction problems}
\author{Satoshi Murai\thanks{Research of SM is partially\textsl{} supported by KAKENHI 21K03190.}\\
\small Department of Mathematics, Faculty of Education\\[-0.8ex] 
\small Waseda University\\[-0.8ex]  
\small 1-6-1 Nishi-Waseda, Shinjuku, Tokyo 169-8050, Japan\\[-0.8ex]
\small \texttt{s-murai@waseda.jp}
\and
	Isabella Novik\thanks{Research of IN is partially\textsl{} supported by NSF grants DMS-1953815 and DMS-2246399. }\\
	\small Department of Mathematics\\[-0.8ex]
	\small University of Washington\\[-0.8ex]
	\small Seattle, WA 98195-4350, USA\\[-0.8ex]
	\small \texttt{novik@uw.edu}
	\and 
	Hailun Zheng\thanks{Research of HZ is partially\textsl{} supported by NSF grant DMS-2246793.} \\
	\small Department of Mathematics \& Statistics\\[-0.8ex]
	\small University of Houston-Downtown\\[-0.8ex]
	\small One Main Street, Houston, TX 77002, USA \\[-0.8ex]
	\small \texttt{zhengh@uhd.edu}
}

\begin{document}
	\maketitle
		\begin{abstract}
		A conjecture of Kalai asserts that for $d\geq 4$, the affine type of a prime simplicial $d$-polytope $P$ can be reconstructed from the space of affine $2$-stresses of $P$. We prove this conjecture for all $d\geq 5$. We also prove the following generalization: for all pairs $(i,d)$ with $2\leq i\leq  \lceil \frac d 2\rceil-1$, the affine type of a simplicial $d$-polytope $P$ that has no missing faces of dimension $\geq d-i+1$ can be reconstructed from the space of affine $i$-stresses of $P$.  A consequence of our proofs is a strengthening of the Generalized Lower Bound Theorem: it was proved by Nagel that for any simplicial $(d-1)$-sphere $\Delta$ and $1\leq k\leq \lceil\frac{d}{2}\rceil-1$, $g_k(\Delta)$ is at least as large as the number of missing $(d-k)$-faces of $\Delta$; here we show that, for $1\leq k\leq \lfloor\frac{d}{2}\rfloor-1$, equality holds if and only if $\Delta$ is $k$-stacked.
		Finally, we show that for $d\geq 4$, any simplicial $d$-polytope $P$ that has no missing faces of dimension $\geq d-1$ is redundantly rigid, that is, for each edge $e$ of $P$, there exists an affine $2$-stress on $P$ with a non-zero value on $e$.
	\end{abstract}
	
	\section{Introduction} 
	This paper is focused on the following question: what partial information about a simplicial polytope $P$ allows one to reconstruct the affine type of $P$? We also use this paper as an opportunity to strengthen connections between geometric/combinatorial theory of (higher) affine stresses on one hand (as developed by Lee \cite{Lee96} and Tay, White, and Whiteley \cite{Tay-et-al-I,Tay-et-al}) and Macaulay inverse systems \cite[Chapter 21]{Eisenbud}, along with other powerful and well-established results and techniques from commutative algebra, on the other. We defer all definitions until the following sections.
	
	Let $P\subset \R^d$ be a simplicial $d$-polytope with $n$ vertices, which we label by $1,2,\ldots,n$, and let $\partial P$ be the boundary complex of $P$ considered as an abstract simplicial complex. By using translation, if necessary, we always assume that the origin is in the interior of $P$. Let $p(1),\ldots,p(n)\in \R^d$ be the position vectors of vertices of $P$, and let $\RR(P)$ be the associated $(d+1)\times n$ matrix whose $i$th column is the vector $p(i)$ augmented by a one in the last position. A basic observation at the heart of the theory of Gale diagrams (see \cite[Chapter 6]{Ziegler}) asserts that the row space of $\RR(P)$ (equivalently, the null space of $\RR(P)$) determines $P$ up to affine equivalence. Put differently, the space of affine dependencies of vertices of $P$ allows one to reconstruct the affine type of $P$. The fact that the space of affine  dependencies of vertices of $P$ is the space of affine $1$-stresses of $P$ motivated Kalai to propose the following bold conjecture: 
	
	\begin{conjecture} \label{conj:Kalai}
	If $d\geq 4$ and $P$ is a simplicial $d$-polytope that has no missing faces of dimension $d-1$, then one can reconstruct the affine type of $P$ from the graph of $P$ and the space of affine $2$-stresses of $P$.
	\end{conjecture}
	\noindent Conjecture \ref{conj:Kalai} was privately communicated to us and recorded in \cite{NZ-reconstruction}; for a related conjecture, see \cite[Conjecture 7]{Kalai-survey}.
	
	Before proceeding, let us restate the above discussion in the language of commutative algebra. Consider a polynomial ring $R=\R[x_1,\dots, x_n]$, let $I_P\subset R$ be the Stanley--Reisner ideal of $\partial P$, let $\Theta_P=\theta_1,\dots,\theta_d$ be the sequence of linear forms determined by the coordinates of vertices of $P$, and let $\ell=x_1+ \cdots +x_n$. (That is, $\theta_1,\ldots,\theta_d,\ell$ correspond to the rows of $\RR(P)$.) Denote by $M_i$ the degree $i$ part of a graded $R$-module $M$. Then $(\Theta_P,\ell)_1=\big(I_P+(\Theta_P,\ell)\big)_1$ determines $P$ up to affine equivalence. Furthermore, Conjecture \ref{conj:Kalai} posits that if $d\geq 4$ and $P$ is a simplicial $d$-polytope with $n$ vertices such that $I_P$ has no generators of degree $d$,  then one can reconstruct $(\Theta_P,\ell)_1$ from $(I_P)_2$ and $\big(I_P+(\Theta_P,\ell)\big)_2$. The reason Conjecture \ref{conj:Kalai} is equivalent to this algebraic statement is that the space of affine $i$-stresses of $P$ coincides with a certain graded part of the Macaulay inverse system of $I_P+(\Theta_P,\ell)$.
	
	Conjecture \ref{conj:Kalai}  was recently proved by Cruickshank, Jackson and Tanigawa \cite{CJT} in the case that $P$ is a simplicial polytope (or a simplicial sphere) whose vertices have ``generic'' coordinates, and by Novik and Zheng \cite{NZ-AffReconstr} for all simplicial $d$-polytopes that have no missing faces of dimension $\geq d-2$. In addition, Novik and Zheng generalized Conjecture \ref{conj:Kalai} to higher stresses: 
	
	\begin{conjecture} \label{Old Conjecture 1.2}
	Let $d\geq 4$, let $2\leq i\leq d/2$, and let $P$ be a simplicial $d$-polytope. If $P$ has no missing faces of dimension $\geq d-i+1$, then the affine type of $P$ can be reconstructed from the space of affine $i$-stresses of $P$. Equivalently, if $I_P$ has no generators of degree $\geq d-i+2$, then $(\Theta_P,\ell)_1$ is determined by $\big(I_P+(\Theta_P,\ell)\big)_i$.
	\end{conjecture}
	
	Conjecture \ref{Old Conjecture 1.2} was proved in \cite{NZ-AffReconstr} for polytopes that have no missing faces of dimension $\geq d-2i+2$. The goal of this paper is to prove Conjecture \ref{Old Conjecture 1.2} in full generality for all pairs $(i,d)$ with $d>2i\geq 4$ (thus also proving Conjecture \ref{conj:Kalai} for $d\geq 5$); see Theorem \ref{thm: Stress_i to Stress_j}. At the moment, the case of $d=2i$ ($d=4$ in Conjecture \ref{conj:Kalai}) remains open. 
	
	The space of affine $i$-stresses of $P$, $\Stress_i$, is a certain subspace of the $i$-th homogeneous part of $R$. To prove Conjecture \ref{Old Conjecture 1.2}, we establish in Theorem \ref{thm: Stress_i to Stress_j} the following stronger result: if $i\leq \lceil d/2\rceil-1$ and $P$ is a simplicial $d$-polytope that has no missing faces of dimension $\geq d-i+1$, then for all $1\leq k < i$, the space of affine $k$-stresses of $P$ is obtained from $\Stress_i$ by computing all partial derivatives of order $i-k$ of all elements of $\Stress_i$.

	Our proof relies on identifying the space of affine stresses of $P$ with the Matlis dual $N$ of $R/\big(I_P+(\Theta_P,\ell)\big)$ and then showing that if $I_P$ has no generators of degree $\geq d-i+2$, then $N$ is generated by elements of degree $\leq -i$. In fact, we show that for $1\leq k< \lceil \frac{d}{2}\rceil -1$,	the dimension of the degree $k$ part of the socle of $R/\big(I_P+(\Theta_P,\ell)\big)$ is the number of generators of $I_P$ of degree $d-k+1$ (equivalently, it is the number of missing faces of $P$ of dimension $d-k$). The proof of the latter statement is based on 
(the algebraic version of) the $g$-theorem for polytopes \cite{McMullen96,Stanley80} and on a result by Migliore and Nagel \cite{MiglioreNagel} that relates the graded Betti numbers of $R/I_P$ considered as an $R$-module to those of $R/\big(I_P+(\Theta_P,\ell)\big)$ considered as a module over $R/(\Theta_P,\ell)$. Since the $g$-theorem was recently established not only for simplicial $d$-polytopes but for all simplicial $(d-1)$-spheres \cite{Adiprasito-g-conjecture,KaruXiao,PapadakisPetrotou}, our result also applies to simplicial spheres with generic embeddings.\footnote{In fact, the $g$-theorem was established for all $\Z/2\Z$-homology spheres (see \cite[Theorem 1.3]{KaruXiao}). As the set of simplicial spheres forms a subfamily of $\Z/2\Z$-homology spheres which, in turn, forms a subfamily of $\R$-homology spheres, all results in this paper, namely, Theorem \ref{thm: Stress_i to Stress_j}, Proposition \ref{prop:socle and missing faces}, and Corollaries \ref{cor:inequalities on g_k} and \ref{cor:k-stacked}, continue to hold in the generality  of $\Z/2\Z$-homology spheres.} 
	
The proof of Conjecture \ref{Old Conjecture 1.2} leads to two corollaries on the $g$-numbers of simplicial $(d-1)$-spheres that are interesting in their own right. For a simplicial complex $\Delta$, denote by $m_{i}(\Delta)$ the number of missing faces of $\Delta$ of dimension $i$. The $g$-numbers of $\Delta$ are defined by $g_0(\Delta):=1$ and $g_k(\Delta):=h_k(\Delta)-h_{k-1}(\Delta)$ for $k>0$,
where $h_0(\Delta),h_1(\Delta),\dots, h_{\dim\Delta+1}(\Delta)$ are the $h$-numbers of $\Delta$. The function $a\mapsto a^{<k>}$ describes the maximum possible dimension that the degree $k+1$ part of a standard graded $\R$-algebra $A$ can have if $\dim A_k=a$.

\begin{corollary} \label{cor:inequalities on g_k} 
Let $d\geq 4$ and let $\Delta$ be a simplicial $(d-1)$-sphere. Then for all $1\leq k\leq \lceil\frac{d}{2}\rceil -1$, 
$g_k(\Delta)\geq m_{d-k}(\Delta)$. Furthermore,
$0\leq g_{k+1}(\Delta)\leq \big(g_k(\Delta)-m_{d-k}(\Delta)\big)^{<k>}$.
\end{corollary}

\begin{corollary} \label{cor:k-stacked}
Let $\Delta$ be a simplicial $(d-1)$-sphere. Then for $1\leq k\leq \lfloor\frac{d}{2}\rfloor  -1$, $\Delta$ is $k$-stacked if and only if $g_k(\Delta)=m_{d-k}(\Delta)$. Moreover, if $d$ is odd and $\Delta$ is $\frac{d-1}{2}$-stacked, then $g_{\frac{d-1}{2}}(\Delta)=m_{\frac{d+1}{2}}(\Delta)$.
\end{corollary}

 We remark that Corollary \ref{cor:inequalities on g_k} provides a strengthening of the (numerical part of the) $g$-theorem asserting that for all $1\leq k\leq \frac{d}{2} -1$ and any simplicial $(d-1)$-sphere $\Delta$, $0\leq g_{k+1}(\Delta)\leq \big(g_k(\Delta)\big)^{<k>}$. After this paper was submitted, we realized that Corollary \ref{cor:inequalities on g_k}  is not new: it is part of \cite[Corollary 4.6(a)]{Nagel} due to Nagel; in fact, the proof that we provide here for completeness closely follows the proof of Nagel. It is also worth pointing out that the inequality $g_k(\Delta)\geq m_{d-k}(\Delta)$ of Corollary \ref{cor:inequalities on g_k} is similar in flavor to \cite[Corollary 6.5]{Adiprasito-toric}.

Simplicial $(d-1)$-spheres that are $k$-stacked are of special interest in connection with the Lower Bound Theorem (for $k=1$) \cite{Kalai87} and the Generalized Lower Bound Theorem (for any $1\leq k\leq \frac d 2 -1$) \cite{MuraiNevo2013}. Specifically, \cite{MuraiNevo2013} (combined with the $g$-theorem for spheres) proves that if $1\le k\leq \frac d 2-1$, then a simplicial $(d-1)$-sphere $\Delta$ is $k$-stacked if and only if $g_{k+1}(\Delta)=0.$\footnote{If $d$ is odd, then $g_{(d+1)/2}(\Delta)=0$ for {\em any} simplicial $(d-1)$-sphere $\Delta$. Thus, Murai--Nevo's characterization of $k$-stacked spheres cannot be extended to the case of $k=(d-1)/2$.} The characterization of $k$-stacked spheres given in Corollary \ref{cor:k-stacked} appears to be new. Both Corollaries \ref{cor:inequalities on g_k} and \ref{cor:k-stacked} make progress on \cite[Problem 1]{Kalai-survey} that asks for a characterization of possible vectors $(m_1,m_2,\ldots,m_d)$ that arise from simplicial $(d-1)$-spheres.

The reader may have noticed that in contrast to Conjecture \ref{conj:Kalai}, knowing the graph of $P$  is not part of the assumptions of the $i=2$ case of Conjecture \ref{Old Conjecture 1.2}. The fact that the graph of $P$ is not needed in this case is an easy consequence of the following result, which we prove here (previously this was only known for spheres with generic embeddings \cite{Z-rigidity}):

\begin{theorem} \label{thm:edges in stresses}
 If $d\geq 4$ and $P$ is a simplicial $d$-polytope that has no missing faces of dimension $\geq d-1$, then every edge of $P$ participates in some affine $2$-stress on $P$, that is, for every edge $e$ of $P$ there exists an affine $2$-stress on $P$ whose value on $e$ is non-zero.
\end{theorem}

This theorem solves the $i=2$ case of \cite[Conjecture 3.3]{NZ-AffReconstr}. Its significance from the rigidity theory point of view is that if $d\geq 4$ and $P$ is a simplicial $d$-polytope that has no missing faces of dimension $\geq d-1$, then the graph of $P$ is redundantly rigid, that is, the graph of $P$ is infinitesimally rigid, and for every edge $e$ of $P$, the graph of $P$ with $e$ deleted is also infinitesimally rigid. Its significance from the algebraic point of view is that if $I_P$ has no generators of degree $\geq d$, then the set of degree two squarefree monomials contained in $I_P+ (\Theta_P,\ell)$ is equal to the set of degree two squarefree monomials contained in $I_P$.

The structure of the rest of the paper is as follows. In Section 2, we provide basic definitions pertaining to simplicial polytopes and simplicial spheres, discuss Stanley--Reisner rings of simplicial complexes and the $g$-theorem for spheres, review some background related to Macaulay's inverse systems and affine stresses as well as to several other commutative algebra notions such as socles and graded Betti numbers. Section 3 is devoted to the proof of Conjecture \ref{Old Conjecture 1.2} for $d>2i$; see Theorem \ref{thm: Stress_i to Stress_j}. In Section 4, we verify Corollaries \ref{cor:inequalities on g_k} and \ref{cor:k-stacked}. Finally, in Section 5, we prove Theorem \ref{thm:edges in stresses}. Section~5 relies on several standard tools from rigidity theory of frameworks, and can be read independently from the rest of the paper.

\section{Setting up the stage}
\subsection{Simplicial polytopes and simplicial spheres}
One of the main objects considered in this paper is the class of simplicial polytopes. A {\em polytope} $P\subset \R^d$ is the convex hull of finitely many points in $\R^d$. A (proper) {\em face} of $P$ is the intersection of $P$ with a supporting hyperplane of $P$. The dimension of $P$ is defined as the dimension of the affine hull of $P$, and we say that $P$ is a {\em $d$-polytope} if the dimension of $P$ is $d$. All (proper) faces of a $d$-polytope are by themselves (smaller-dimensional) polytopes. We refer to $0$-faces as {\em vertices}, $1$-faces as {\em edges}, and $(d-1)$-faces as {\em facets}. A polytope $P$ is a (geometric) {\em simplex} if the vertices of $P$ are affinely independent. A polytope $P$ is {\em simplicial} if all facets of $P$ are simplices. 

Related to simplicial polytopes  is the notion of simplicial complexes. An (abstract) {\em simplicial complex} $\Delta$ with (a finite) vertex set $V=V(\Delta)$ is a collection of subsets of $V$ that is closed under inclusion and contains all singletons. The elements of $\Delta$ are {\em faces}. We say that $F\in \Delta$ is a face of {\em dimension $i$} or an {\em $i$-face} if $|F|=i+1$. The {\em dimension} of $\Delta$ is the maximum dimension of faces of $\Delta$. As in the case of polytopes, $0$-faces are called {\em vertices}, $1$-faces are {\em edges}, and maximal under inclusion faces are {\em facets}. For brevity, when referring to vertices and edges, we usually write $v$ and $uv$ instead of $\{v\}$ and $\{u,v\}$, respectively. 

A subset $F$ of $V$ is a {\em missing face} of $\Delta$ if $F$ is not a face of $\Delta$ but all proper subsets of $F$ are faces of $\Delta$; we say that a missing face has dimension $i$ if $|F|=i+1$. If $\Delta$ is $(d-1)$-dimensional, then the missing faces of $\Delta$ of dimension $d-1$ are often called {\em missing facets}. It is worth remarking that the collection of all missing faces of $\Delta$ together with $V(\Delta)$ uniquely determines $\Delta$.

If $\Delta$ is a $(d-1)$-dimensional simplicial complex and $0\leq i\leq d-1$, then the {\em $i$-skeleton} of $\Delta$, $\skel_i(\Delta)$, is a subcomplex of $\Delta$ consisting of all faces of $\Delta$ of dimension at most $i$. The $1$-skeleton of $\Delta$ is also called the {\em graph} of $\Delta$. If $W\subset V(\Delta)$ then the subcomplex of $\Delta$ {\em induced} by $W$ is $\Delta_W:=\{F\in \Delta : F\subseteq W\}$. Two other important types of subcomplexes of $\Delta$ are links and stars. If $F$ is a face of $\Delta$ then the {\em link} of $F$ and the {\em star} of $F$ are defined as follows:
\[\st_\Delta(F)=\st(F):=\{G\in\Delta : F\cup G\in\Delta\} \mbox{ and } \lk_\Delta(F)=\lk(F):=\{G\in\st(F) : F\cap G=\emptyset\}. \]

If $\Delta$ is a simplicial complex and $\Gamma$ is a subcomplex of $\Delta$, we let $(\Delta,\Gamma)$ denote the corresponding {\em relative simplicial complex}: its faces are precisely the faces of $\Delta$ not contained in $\Gamma$. The {\em dimension} of  $(\Delta,\Gamma)$ is defined as the maximum dimension of its faces.

Associated with a simplicial complex $\Delta$ is a topological space $\|\Delta\|$ (realized as a subset of $\R^k$ for some $k$ that is usually larger than $\dim\Delta$) called the {\em geometric realization} of $\Delta$: it contains a geometric $i$-simplex for each $i$-face of $\Delta$. A simplicial complex $\Delta$ is called a {\em simplicial $(d-1)$-sphere} if $\|\Delta\|$ is homeomorphic to a $(d-1)$-sphere. Similarly, $\Delta$ is a {\em simplicial $(d-1)$-ball} if $\|\Delta\|$ is homeomorphic to a $(d-1)$-ball. 

Every simplicial $d$-polytope $P$ gives rise to a simplicial $(d-1)$-sphere $\partial P$ called the {\em boundary complex} of $P$: the faces of $\partial P$ are the vertex sets of all proper faces of $P$ (i.e., including the empty face, but not including $P$ itself). Furthermore, for a face $\tau$ of $\partial P$, $\lk_{\partial P}(\tau)$ is the boundary complex of a simplicial polytope. When $\tau=v$ is a vertex, one such polytope is obtained by intersecting $P$ with a hyperplane that has $v$ on one side and all other vertices of $P$ on the other. The resulting polytope, $P/v$, is called a {\em vertex figure} of $P$ at $v$ or a {\em quotient polytope of $P$ by $v$}. For a face $\tau$ of positive dimension, a {\em quotient of $P$ by $\tau$}, $P/\tau$, is obtained by iteratively taking vertex figures of polytopes at the vertices of $\tau$.

Simplicial spheres and balls are special cases of $\R$-homology spheres and $\R$-homology balls, respectively. A $(d-1)$-dimensional simplicial complex $\Delta$ is  a {\em homology $(d-1)$-sphere} if for every face $F$ of $\Delta$ (including the empty face), the (simplicial)  $\R$-homology of $\lk_\Delta(F)$, $\tilde{H}_*(\lk_\Delta(F);\R)$, coincides with that of a $(d-|F|)$-sphere. A $(d-1)$-dimensional simplicial complex $\Delta$ is called a {\em homology $(d-1)$-ball} if (1) the $\R$-homology of $\Delta$ coincides with that of a $(d-1)$-ball, (2) for each nonempty face $F\in \Delta$, the link of $F$ has the $\R$-homology of either a $(d-|F|)$-sphere or a $(d-|F|)$-ball, and (3) the {\em boundary complex} of $\Delta$, i.e., $\partial \Delta:=\big\{F\in\Delta : \tilde{H}(\lk_\Delta(F);\R)=0\big\},$ is a homology $(d-2)$-sphere. The boundary complex of a homology $(d-1)$-ball $\Delta$ can alternatively be described as the collection of all $(d-2)$-faces of $\Delta$ that are contained in exactly one facet, together with their subsets. A face $F$ of $\Delta$ is called a {\em boundary face} if $F\in\partial \Delta$ and $F$ is an {\em interior face} if $F\notin\partial\Delta$. That is, the interior faces of $\Delta$ form the relative simplicial complex $(\Delta,\partial\Delta)$.

Let $\Delta$ be a simplicial $(d-1)$-sphere and $0\leq k\leq d$. We say that $\Delta$ is {\em $k$-stacked} if there exists a homology $d$-ball $B$ such that $\Delta=\partial B$ and all interior faces of $B$ are of dimension $\geq d-k$. In other words, $\Delta=\partial B$ and the relative complex $(B,\partial B)$ has no faces of dimension $\leq d-k-1$. Any $0$-stacked sphere is the boundary complex of a simplex; $1$-stacked spheres are known in the literature simply as stacked spheres. 

Let $\Delta$ be either a $(d-1)$-dimensional simplicial complex or a simplicial $d$-polytope. For $-1\leq i \leq d-1$, we let $f_i(\Delta)$ denote the number of $i$-faces of $\Delta$. The {\em $f$-vector} of $\Delta$ is $f(\Delta)=(f_{-1}(\Delta), f_0(\Delta), \ldots, f_{d-1}(\Delta)).$ The {\em $h$-vector} of $\Delta$, $h(\Delta)=(h_0(\Delta),h_1(\Delta), \ldots,h_d(\Delta)),$ is defined by the following polynomial relation: $\sum_{i=0}^d h_i(\Delta) t^{d-i}= \sum_{i=0}^d f_{i-1}(\Delta) (t-1)^{d-i}$. Thus, the $h$-vector is obtained from the $f$-vector by an invertible linear transformation and $h_0(\Delta)=f_{-1}(\Delta)=1$. We also define the {\em $g$-numbers} of $\Delta$ by $g_0(\Delta):=1$ and $g_i(\Delta):=h_i(\Delta)-h_{i-1}(\Delta)$ for $0<i\leq \lceil d/2 \rceil$.  For a relative simplicial complex $(\Delta,\Gamma)$, one defines the $f$- and $h$-vectors using the same formulas. The only thing to note is that if $\Gamma$ is not the void complex, then $h_0(\Delta,\Gamma)=f_{-1}(\Delta,\Gamma)=0$.

\subsection{Stanley--Reisner rings and the $g$-theorem}
Let $\Delta$ be a $(d-1)$-dimensional simplicial complex with vertex set $V(\Delta)$ of size $n$ which we identify with $[n]:=\{1,2,\ldots,n\}$. Let $R:=\R[x_1,\ldots,x_n]$ be a polynomial ring with $n$ variables. The {\em Stanley--Reisner ideal} of $\Delta$ is the ideal of $R$ generated by non-faces of $\Delta$ (equivalently, by missing faces of $\Delta$):
\[I_\Delta=(x_{j_1}x_{j_2}\cdots x_{j_k} : \{j_1,\ldots,j_k\}\notin\Delta).\]
The {\em Stanley--Reisner ring} (or {\em face ring}) of $\Delta$ is the quotient $R/I_\Delta$. This is a graded ring. The Hilbert series of $R/I_\Delta$ is given by $\big(\sum_{i=0}^d h_i(\Delta) t^i\big)/(1-t)^d$; see \cite[Theorem II.1.4]{Stanley96}. In particular, the Krull dimension of $R/I_\Delta$ is $d$. A sequence of $d$ linear forms $\theta_1,\ldots,\theta_d\in R$ is called a {\em linear system of parameters} of $R/I_\Delta$ ({\em l.s.o.p.}~for short) if the ring $R/\big(I_\Delta+(\theta_1,\ldots,\theta_d)\big)$ is a finite-dimensional $\R$-vector space.

A {\em $d$-embedding} of $\Delta$ is any map $p: V(\Delta)=[n] \to \R^d$. We consider two types of embeddings: natural and generic. A $d$-embedding $p$ is called {\em generic} if the multiset of coordinates of the points $p(v)$, $v\in [n]$, is algebraically independent over $\Q$. In the case that $\Delta=\partial P$ is the boundary complex of a simplicial $d$-polytope $P$, one can also consider $p: [n]\to \R^d$ given by the position vectors of vertices of $P$; such $p$ is called a {\em natural} embedding. 
In this case, for $F=\{i_1,\ldots,i_k\}\in\partial P$, we write $[p(i_1),\ldots, p(i_k)]$ as a shorthand for the corresponding face $\conv(p(i_1),\ldots,p(i_k))$ of $P$. 

A $d$-embedding $p$ gives rise to $d$ linear forms in $R$: $\theta_i:=\sum_{v=1}^n p(v)_ix_v$ (for $i=1,\ldots,d$), where $p(v)_i$ is the $i$th coordinate of $p(v)$. Let $\Theta =\Theta(p)$ be the set of these $d$ forms. Assume that $(\Delta, p)$ is either the boundary complex a simplicial $d$-polytope $P$ (that contains the origin in the interior) with its natural embedding or a simplicial $(d-1)$-sphere with a generic embedding. Then $\Theta =\Theta(p)$ is an l.s.o.p of $R/I_\Delta$ (this follows from \cite[Theorem III.2.4]{Stanley96}); call $\Theta$ the {\em l.s.o.p.~determined by $p$}.  

The Stanley--Reisner ring of a simplicial $(d-1)$-sphere $\Delta$ is {\em Cohen--Macaulay} \cite{Reisner}; hence any l.s.o.p.~$\Theta$ of $R/I_\Delta$ is a {\em regular sequence}, that is,  for any $1\leq i\leq d$, $\theta_i$ is a non-zero-divisor on $R/\big(I_\Delta+(\theta_1,\ldots,\theta_{i-1})\big)$. Consequently, $\dim_\R \big(R/(I_\Delta+(\Theta))\big)_i=h_i(\Delta)$ for all $0\leq i\leq d$. For instance, if $\Delta$ is the boundary complex of a $2$-simplex, then  $R/I_\Delta=\R[x_1,x_2,x_3]/(x_1x_2x_3)$. Furthermore, if  the position vectors of vertices of $\Delta$ in $\R^2$ are $p(1)=(1,0)$, $p(2)=(0,1)$, and $p(3)=(-1,-1)$, then $\theta_1=x_1-x_3$ and $\theta_2=x_2-x_3$. In this case $R/(I_\Delta+(\Theta))$ is isomorphic to $R[x]/(x^3)$. This agrees with an easy fact that  $h(\Delta)=(1,1,1)$.

A much harder result, known as the algebraic version of the $g$-theorem, asserts that if $(\Delta, p)$ is as in the previous paragraphs and $\Theta=\Theta(p)$, then $\ell=\sum_{v=1}^n x_v$ is a non-zero-divisor on degree $\leq \lceil d/2\rceil-1$ parts of $R/\big(I_\Delta+(\Theta)\big)$; {for this reason, we call $\ell$ the {\em canonical Lefschetz element.} More precisely, the following holds (see \cite{McMullen96,Stanley80} for the case of polytopes and \cite{Adiprasito-g-conjecture,KaruXiao,PapadakisPetrotou}, e.g., \cite[Theorem 1.3]{KaruXiao}, for the case of spheres).

\begin{theorem} \label{alg-g-thm}
Let $(\Delta,p)$ be either the boundary complex of a simplicial $d$-polytope $P$ with its natural embedding or a simplicial $(d-1)$-sphere with a generic embedding, let $\Theta=\Theta(p)$ be the l.s.o.p.~determined by $p$, and let $\ell=\sum_{v=1}^n x_v$. Then the map 
$$\cdot \ell: \big(R/(I_\Delta+(\Theta)\big)_i \to \big(R/(I_\Delta+(\Theta)\big)_{i+1}$$ is injective for all $0\leq i\leq \lceil d/2\rceil -1$ and surjective for all $\lfloor d/2\rfloor \leq i\leq d-1$. In particular, $\dim_\R  \big(R/\big(I_\Delta+(\Theta, \ell)\big)\big)_i=g_i(\Delta)$ for all $i\leq \lceil d/2\rceil$.
\end{theorem}

 Macaulay's theorem \cite[Theorem II.2.2]{Stanley96} provides a characterization of Hilbert functions of finitely generated standard graded $\R$-algebras. This characterization is given in terms of a certain function $a \mapsto a^{\langle i \rangle}$. To define it, note that for every pair $a, i$ of positive integers, there is a unique representation of $a$ in the form $a=\binom{a_i}{i}+ \binom{a_{i-1}}{i-1}+\cdots+\binom{a_j}{j}$, where $a_i>a_{i-1}>\cdots>a_j\geq j\geq 1$. The function $a \mapsto a^{\langle i \rangle}$ is then defined as follows: 
\[a^{\langle i \rangle}:=\binom{a_i+1}{i+1}+ \binom{a_{i-1}+1}{i}+\cdots+\binom{a_j+1}{j+1} \quad \mbox{and} \quad 0^{\langle i \rangle}:=0.\] Theorem \ref{alg-g-thm} along with Macaulay's theorem applied to $R/(I_\Delta+(\Theta,\ell))$ yields: 

\begin{theorem} \label{num-g-thm}
Let $\Delta$ be a simplicial $(d-1)$-sphere. Then $0\leq g_{i+1}(\Delta)\leq (g_i(\Delta))^{\langle i \rangle}$ for all $1\leq i\leq d/2-1$ and if $d$ is odd, then $g_{\lceil d/2\rceil}(\Delta)=0$.
\end{theorem}

The case of $g_{i+1}(\Delta)=0$ for some $i$ is of particular interest. A result of Murai and Nevo \cite[Theorem 1.3]{MuraiNevo2013} (along with the algebraic version of the $g$-theorem) asserts that for $0\leq i\leq d/2-1$, a simplicial $(d-1)$-sphere $\Delta$ is $i$-stacked if and only if $g_{i+1}(\Delta)=0$.

\subsection{Macaulay's inverse systems and affine stresses}

We start this subsection by defining Macaulay's inverse systems.
We mostly follow the notation used in \cite[Section 2]{ER2023}.
Recall that $R= \mathbb R[x_1,\dots,x_n]$.
Let $D=\mathbb R[y_1,\dots,y_n]$ be another polynomial ring with variables $y_1,\dots,y_n$ each of degree $-1$. 
We define an action of $R$ on $D$ by
\[
\begin{array}{cccc}
\circ : &R \times D &\to& D\\
& \mbox{\rotatebox{90}{$\in$}} & & \mbox{\rotatebox{90}{$\in$}}\\
& (f,g) & \mapsto & f(\partial_{y_1},\dots,\partial_{y_n}) \cdot g,
\end{array}
\]
where $\partial_{y_i}=\frac \partial {\partial y_i}$ is the partial derivative w.r.t.~$y_i$.
This action gives an $R$-module structure to $D$, and for each ideal $I \subset R$, the set
\[
I^\perp =\{F \in D :  g \circ F=0\ \mbox{ for all }g\in I\}
\]
becomes an $R$-submodule of $D$.
This submodule $I^\perp$ of $D$ is called the {\em inverse system of $I$}.
If $I$ is a homogeneous ideal, then $I^\perp$ is a graded $R$-module with $(I^\perp)_k= \{F \in D_k : I_{-k} \circ F=0\}$.

Inverse systems can be considered as a special case of the Matlis duality.
For a graded $R$-module $M$, its {\em Matlis dual} $M^\vee$ is the module whose $k$th graded component (for all $k\in \Z$) is $\mathrm{Hom}_{\mathbb R}(M_{-k},\mathbb R)$ and, for $f \in (M^\vee)_k$ and $g \in R$,
$gf$ is the composition $M_{-k-\deg g} \stackrel {\cdot g} \to M_{-k} \stackrel {f} \to \R$.
(See \cite[\S 3.6]{BHbook}).
It is known that $I^\perp$ is isomorphic to the Matlis dual of $R/I$.
To see this, observe that
for monomials $x^{\mathbf a}=x_1^{a_1}\cdots x_n^{a_n} \in R_k$ and $y^{\mathbf b}=y_1^{b_1} \cdots y_n^{b_n} \in D_{-k}$,
one has
\[	
x^{\mathbf a} \circ y^{\mathbf b}
=
\begin{cases}
b_1!  \cdots  b_n!
& \mbox{ if $a_k = b_k$ for all $k$},\\
0 & \mbox{ otherwise}.
\end{cases}
\]
This shows that the action $\circ$ induces a non-singular $\R$-bilinear pairing
\[
R_k \times D_{-k} \to \R,  \ \ (f,g) \mapsto f \circ g,
\]
which gives an isomorphism from $D_{-k}$ to $\mathrm{Hom}_{\R} (R_k,\R)$.
This isomorphism, in turn, induces a non-singular $\R$-bilinear pairing $(R/I)_k \times (I^\perp)_{-k} \to \mathbb R$
and an isomorphism from $(I^\perp)_{-k}$ to $\mathrm{Hom}_{\R} ((R/I)_k,\R) \cong (R/I)^\vee_{-k}$.
It is not difficult to see that this isomorphism is, in fact, an isomorphism of $R$-modules, and so $I^\perp \cong (R/I)^\vee$.
(A more algebraic exposition of this isomorphism is as follows: using the fact that the injective hull $E_R(R/\m)$ is isomorphic to $D$ as an $R$-module,
one obtains $I^\perp \cong \mathrm{Hom}_R(R/I,D)\cong \mathrm{Hom}_R(R/I,E_R(R/\m)) \cong (R/I)^\vee$.)
For more information on inverse systems,
see \cite{ER2017,ER2023} and references therein.

We now turn our discussion to affine stresses. Classical stresses were developed as part of the framework rigidity theory; see \cite{AsimowRothI,AsimowRothII}. Higher stresses were developed in works of Lee \cite{Lee96} and Tay, White, and Whiteley \cite{Tay-et-al-I,Tay-et-al}. Lee's definition can be (re)stated as follows. Let $\Delta$ be a simplicial complex (not necessarily of dimension $d-1$) with vertex set $[n]$, and let $p$ be a $d$-embedding of $\Delta$, i.e., a map $p:[n]\to\R^d$. As in Section 2.2, let $\theta_1,\ldots,\theta_d$ be the sequence of $d$ linear forms {\em determined by $p$}, that is, $\theta_i=\sum_{v\in [n]} p(v)_ix_v$. Denote this sequence by $\Theta(p)$ or simply by $\Theta$ (when $p$ is fixed), and let $\ell=\sum_{v\in[n]} x_v$. The space of {\em affine stresses} of $(\Delta,p)$ is defined as 
\[
\Stress(\Delta,p):=\left\{ f \in R \mid
g(\partial_{x_1},\dots,\partial_{x_n}) \cdot f = 0 \mbox{ for any } g(x_1,\dots,x_n) \in \big(I_\Delta+(\Theta(p),\ell)\big)\right\}
\]
and the space of {\em affine $i$-stresses} of $(\Delta,p)$, denoted $\Stress_i(\Delta,p)$, is the degree $i$ part of $\Stress(\Delta,p)$. In other words, an affine $i$-stress on $(\Delta,p)$ is a homogeneous polynomial $\lambda=\sum_\mu \lambda_\mu \mu\in R$ of degree $i$ such that
(a) every non-zero term $\lambda_\mu \mu$ of $\lambda$ is supported on a face of $\Delta$, i.e., $\supp(\mu):=\{k\in[n] : x_k | \mu\}\in \Delta$, 
(b) $\partial_{\theta_i}(\lambda)=0$ for all $i=1,\ldots,d$, and (c) $\partial_\ell(\lambda)=0$. 
Obviously, $\Stress(\Delta,p)$ is isomorphic to $(I_\Delta+(\Theta(p),\ell))^\perp$
with $\Stress_i(\Delta,p)\cong (I_\Delta+(\Theta(p),\ell))^\perp_{-i}$.

Abusing notation, we write $\lambda_F$ instead of $\lambda_\mu$ when $\mu$ is a squarefree monomial with $\supp(\mu)=F$ and refer to $\lambda_F$ as the {\em weight} assigned  to $F$ by $\lambda$. We also say that a face $G$ \emph{participates in} $\lambda$ if $\lambda_F\neq 0$ for some face $F\supseteq G$. By results of Lee \cite[Theorems 9, 11]{Lee96}, an affine $i$-stress $\lambda$ is uniquely determined by its squarefree part $(\lambda_F)_{F\in\Delta}$. In addition, the squarefree part of $\lambda$ has a particularly nice geometric interpretation \cite[Theorem 10]{Lee96}: the weights assigned to $(i-1)$-faces satisfy certain balancing conditions at $(i-2)$-faces. When $i=1$, the balancing condition at the empty face simply says that $(\lambda_v)_{v\in [n]}$ is an affine dependence of the point configuration $(p(v) : v\in [n])$. When $i=2$, the conditions are at vertices (we let $\mathbf{0}\in\R^d$ denote the origin): 
\begin{equation}\label{eq:balancing}\mbox{for every $u\in [n]$, } \sum_{v: uv\in\Delta} \lambda_{uv}(p(u)-p(v))=\mathbf{0};\end{equation}
 furthermore, for every $u\in [n]$, the coefficient of $x_u^2$ in a $2$-stress $\lambda$ is given by $-\frac{1}{2}\sum_{w:uw\in\Delta}\lambda_{uw}$.

To close this subsection, we restate the last part of Theorem \ref{alg-g-thm} in the language of stresses.

\begin{theorem} \label{g-thm-stresses}
Let $(\Delta,p)$ be either the boundary complex of a simplicial $d$-polytope with its natural embedding or a simplicial $(d-1)$-sphere with a generic embedding. Then for all $i\leq \lceil d/2\rceil$, $\dim_\R \Stress_i(\Delta,p)=g_i(\Delta)$.
\end{theorem}

\subsection{Generators, socles and graded Betti numbers}
In this subsection we discuss several other commutative algebra definitions we will need. Denote by $\m=\m_R:=(x_1,\ldots,x_n)$ the {\em maximal ideal} of $R$ (also known as the {\em irrelevant ideal} of $R$).
For a finitely generated graded $R$-module $M$, let 
$$\mu_k(M)=\dim_{\mathbb R} (M/\m_R M)_k.$$
In other words, $\mu_k(M)$ is the number of degree $k$ elements in a minimal generating set of $M$.
We also introduce a quantity dual to $\mu_k(M)$.
The {\em socle} of a graded $R$-module $M$ is \[
\soc^R M= 0:_M \m =\{y\in M : \m y=0\}.
\]
Let $r_k(M)=\dim_{\mathbb R} (\soc^R M)_k$.
The numbers $\mu_k(-)$ and $r_k(-)$ are dual to each other in the sense that
\begin{align}
\label{2-0}
r_k(M)=\mu_{-k}(M^\vee);
\end{align}
see, for instance, \cite[Proposition 3.2.12(d)]{BHbook}.
We will, in fact, use a set of algebraic invariants that contains both $\mu_k(-)$ and $r_k(-)$.
For a finitely generated graded $R$-module $M$,
the number
\[\beta_{i,j}^R(M)=\dim_\R \big(\tor^R_i(M,\R)\big)_j 
\]
is called the {\em $(i,j)$th graded Betti number (over $R$)};
here we identify $\R$ with $R/\m$.
Alternatively, by \cite[Proposition A.2.2]{HHbook}, the
graded Betti numbers of $M$ can be defined as the exponents appearing in the minimal free $R$-resolution $F_\bullet$ of $M$ where the indexing is chosen as follows:
\[
\textstyle
F_\bullet:
\cdots \longrightarrow
\bigoplus_{j \in \mathbb Z} R(-j)^{\beta^R_{1,j}(M)} \longrightarrow
\bigoplus_{j \in \mathbb Z} R(-j)^{\beta^R_{0,j}(M)} \longrightarrow M \longrightarrow 0.
\]
Here $R(-j)$ is the graded module $R$ with grading shifted by degree $-j$.

The following lemma summarizes a few basic properties of graded Betti numbers.

\begin{lemma} \label{lem:basic-prop-of-Betti}
Let $M$ be a finitely generated graded $R$-module and let $I$ be a homogeneous ideal of $R$.
Then
\begin{itemize}
\item[(1)] $\beta^R_{0,k}(M)=\mu_k(M)$.
\item[(2)] $\beta^R_{n,n+k}(M)=r_k(M)$.
\item[(3)] $\beta^R_{i,j}(I)=\beta^R_{i+1,j}(R/I)$.
\end{itemize}
\end{lemma}

Statements (1) and (3) are immediate consequences of the definition of graded Betti numbers in terms of minimal free resolutions.
Statement (2) follows from the fact that $\tor^R_i(M,\R)$ is isomorphic to the $i$th Koszul homology $H_i(\mathbf x,M)$ of $\mathbf {x}=x_1,\dots,x_n$ with coefficients in $M$ (see \cite[Corollary A.3.5]{HHbook})
and the fact that
$H_n(\mathbf x,M)$ is isomorphic to $\soc^R(M)$ with degrees shifted by $n$ (see \cite[p.\ 268]{HHbook}).

\begin{example} To illustrate Lemma \ref{lem:basic-prop-of-Betti} and other notions introduced in this section, consider $R=\R[x]$ and $M=R/(x^3)$. Then $\soc^R M=M_2$, and so $r_2(M)=1$ while $r_j(M)=0$ for all $j \neq 2$. Also, $M$ is generated by any nonzero element of $M_0$; hence $\mu_0(M)=1$ while $\mu_j(M)=0$ for all $j\neq 0$. 
Finally, $n=1$ and the minimal free $R$-resolution of $M$ is given by $0\to R(-3)^1 \to R^1 \to M\to 0$. Thus, $\beta_{0,0}(M)=1=\mu_0(M)$, $\beta_{1,3}(M)=1=r_2(M)$, and the rest of the Betti numbers are zeros. 
\end{example}

In the proof of Theorem \ref{thm: Stress_i to Stress_j}, we will need to compute the dimensions of graded components of $\soc^R(R/(I_\Delta+(\Theta(p),\ell)))$. To do so, in addition to Lemma \ref{lem:basic-prop-of-Betti}, we will rely on the following results.

Let $\Delta$ be a simplicial complex with $n$ vertices. The number $m_k(\Delta)$ is one of the graded Betti numbers of the Stanley--Reisner ring $R/I_\Delta$.
Indeed,
missing $k$-faces of $\Delta$ correspond to degree $k+1$ elements of the minimal monomial generators of $I_\Delta$, and so $m_k(\Delta)=\beta_{0,k+1}^R(I_\Delta)$. Thus,
\begin{align}
\label{2-1}
m_k(\Delta)=\beta_{1,k+1}^R(R/I_\Delta) \ \ \mbox{ for all $k$.}
\end{align}

Recall that $f \in R_k$ is a {\em non-zero-divisor} on a graded $R$-module $M$ if $fg \ne 0$ for any $0 \ne g \in M$; equivalently,
if the multiplication map $\cdot f : M_i \to M_{i+k}$ is injective for all $i$.
The following lemma asserts that the graded Betti numbers do not change when we pass to the quotient by a non-zero-divisor. (In fact, below we will prove a more general result.)

\begin{lemma} 
\label{lem:nonzerodiv}
Let $M$ be a finitely generated graded $R$-module,
$f \in R_1$ a linear form, and $\overline R=R/f R$.
If $f$ is a non-zero-divisor on $M$, then
\[
\beta_{i,j}^{\overline R}(M/f M) =
\beta_{i,j}^R(M)\ \ \ \mbox{ for all }i,j.
\]
\end{lemma}

We note that $\overline R$ is isomorphic to the polynomial ring with $n-1$ variables.
Hence the left-hand side of the above equation can be considered as a graded Betti number of a module over a polynomial ring with one less variable.
Lemma \ref{lem:nonzerodiv} in particular implies that
if $R/I_\Delta$ is Cohen--Macaulay and $\Theta$ is its l.s.o.p., then
\begin{equation} \label{eq:*}
\beta_{i,j}^{R/(\Theta) R}\big(R/\big(I_{\Delta}+(\Theta)\big)\big) = \beta_{i,j}^R(R/I_{\Delta})
\ \ \mbox{ for all }i,j.
\end{equation}
This result will be used in the next section.

In fact, we will need a strengthening of Lemma \ref{lem:nonzerodiv}.
One may notice that the canonical Lefschetz element $\ell=\sum_{v=1}^n x_v$ of Theorem \ref{alg-g-thm} behaves like
a non-zero-divisor up to degree $i \leq \lceil \frac d 2 \rceil -1$. It turns out that the injectivity of multiplication by $\ell$ in certain (rather than in all) degrees is enough to get a good handle on how the graded Betti numbers change when we consider the quotient by $\ell$. This theory was developed in \cite{MiglioreNagel} and applied to the study of missing faces of simplicial polytopes in \cite{Nagel}. We do not need the full strength of the result established in \cite{MiglioreNagel}; instead, we summarize in Lemma \ref{lem:MN} parts that will be used in this paper.

For a graded $R$-module $M$ and $f \in R_1$,
let $0:_M f= \{ g \in M:  fg=0\}$.
It follows from \cite[Lemma 8.3]{MiglioreNagel} that there is an exact sequence
\begin{align*}
\cdots \longrightarrow
\tor_{i-1}^R\big((0:_M f),\R\big)_{j-1}
\longrightarrow
\tor_i^R(M,\R)_j
\longrightarrow
\tor_i^{R/f R}(M/f M,\R)_j
\longrightarrow \cdots\\
\longrightarrow
\tor_{0}^R\big((0:_M f),\R\big)_{j-1}
\longrightarrow
\tor_1^R(M,\R)_j
\longrightarrow
\tor_1^{R/f R}(M/f M,\R)_j\longrightarrow 0.
\end{align*}
If the multiplication map $\cdot f :M_k \to M_{k+1}$ is injective when $k \leq s$,
then $(0:_M f)_k=0$ for $k \leq s$,
which implies that $\tor_i^R((0:_M f),\R)_{i+j}=0$ for all $i$ and $j \leq s$ (see \cite[Proposition 12.3]{Peevabook}).
The above long exact sequence then yields the following statement.

\begin{lemma}
\label{lem:MN}
Let $M$ be a finitely generated graded $R$-module and $f\in R_1$.
Suppose that the multiplication map $\cdot f: M_k \to M_{k+1}$ is injective for $k \leq s$. Then for all $i$,
\begin{itemize}
\item[(1)] $\beta_{i,i+k}^{R/f R}(M/fM) =\beta_{i,i+k}^R(M)$ if $k<s$, and  
\item[(2)] $\beta_{i,i+s}^{R/f R}(M/f M) \geq \beta_{i,i+s}^R(M)$.
\end{itemize}
\end{lemma}

In particular, if $\Delta$ and $\Theta$ satisfy the assumptions of Theorem \ref{alg-g-thm}, then Lemma \ref{lem:MN} implies that for all $i$,
\begin{eqnarray}
\label{2-2}
\beta_{i,i+k}^{R/(\Theta,\ell)R}\big(R/\big(I_\Delta+(\Theta,\ell)\big)\big) &=&\beta_{i,i+k}^{R/(\Theta)R}\big(R/\big(I_{\Delta}+(\Theta)\big)\big) \ \ \textstyle
\mbox{ for }k <  \lceil \frac d 2 \rceil -1, \mbox{ and}\\
\label{2-3}
\beta_{i,i+\lceil \frac d 2 \rceil -1}^{R/(\Theta,\ell)R}\big(R/\big(I_\Delta+(\Theta,\ell)\big)\big) &\geq& \beta_{i,i+\lceil \frac d 2 \rceil -1}^{R/(\Theta)R}\big(R/\big(I_{\Delta}+(\Theta)\big)\big).
\end{eqnarray}

Another ingredient we will need is the following duality of the graded Betti numbers of the Stanley--Reisner rings of spheres.

\begin{lemma}
\label{lem:symmetry}
If $\Delta$ is a simplicial $(d-1)$-sphere with $n$ vertices, then
\[
\beta_{i,j}^R(R/I_\Delta)=\beta_{n-d-i,n-j}^R(R/I_\Delta) \mbox{ for all $i,j$}.\]
\end{lemma}

\begin{proof}
We only sketch the proof since the result is well-known to experts. If $\Delta$ is a simplicial $(d-1)$-sphere, then $R/I_\Delta$ is a Gorenstein ring of Krull dimension $d$ and the graded canonical module of $R/I_\Delta$ coincides with $R/I_\Delta$ (see \cite[Theorem 5.6.1 and the proof of Theorem 5.6.2]{BHbook}).
Let 
\[
F_\bullet: F_{n-d}=R(-n) \to \cdots \to F_1 \to F_0 \to 0\]
be a minimal free resolution of $R/I_{\Delta}$.
Then since $R/I_\Delta$ is Cohen-Macaulay,
$G_\bullet=\mathrm{Hom}_{R}(F_\bullet,R(-n))$ is a minimal free resolution of the canonical module of $R/I_\Delta$ (see \cite[Corollary 3.3.9]{BHbook}). That the canonical module of $R/I_\Delta$ is isomorphic to $R/I_\Delta$ implies the desired symmetry.
\end{proof}

\noindent Lemma \ref{lem:symmetry} also follows from Hochster's formula and the Alexander duality.

 \section{Reconstructing stresses from higher stresses}
The goal of this section is to prove Conjecture \ref{Old Conjecture 1.2} on reconstructing the affine type of a simplicial polytope $P$ that has no large missing faces from the space of affine $i$-stresses of $P$. Since the space of affine $1$-stresses of $P$ coincides with the space of affine dependencies of vertices of $P$, to reconstruct the affine type of $P$, it suffices to reconstruct the space of affine $1$-stresses. In fact, we prove the following stronger result: we show that under certain assumptions, the space of affine $i$-stresses determines the space of affine $k$-stresses for {\em all} $1\leq k\leq i$.

\begin{theorem} \label{thm: Stress_i to Stress_j} 
		Let $(\Delta, p)$ be either the boundary complex of a simplicial $d$-polytope with its natural embedding $p$, or a simplicial $(d-1)$-sphere with a generic embedding $p$. Let $1 \leq k<i \leq \lceil \frac{d}{2}\rceil-1$. If $\Delta$ has no missing faces of dimension $\geq d-i+1$, then 
		\begin{equation} \label{eq: Stress_i to Stress_j} \Stress_k(\Delta,p)=\big\{q(\partial_{x_1},\ldots,\partial_{x_d})\cdot \lambda \ : \ \lambda\in \Stress_i(\Delta,p), \ q\in R_{i-k}\big\}.\end{equation}
		Thus, the space of affine $i$-stresses, $\Stress_i(\Delta, p)$, determines the space of affine $k$-stresses, $\Stress_k(\Delta, p)$ for all $1\leq k<i$. In particular, $\Stress_i(\Delta, p)$ determines $p$ up to an invertible affine transformation.
\end{theorem}
		
		Before we proceed, it is informative to unwrap the statement of Theorem \ref{thm: Stress_i to Stress_j}  in the case of $d\geq 5$, $k=1$, and $i=2$. Recall from Section 2.3 that $\lambda$ is an element of $\Stress_2(\Delta,p)$ if and only if $\lambda$ is of the form $\sum_{u\in [n]}\lambda_{u} x_u^2 + \sum_{uv\in\Delta}\lambda_{uv}x_ux_v$ where $(\lambda_{uv})_{uv\in\Delta}$ satisfy the balancing conditions \eqref{eq:balancing} and $\lambda_u=-\frac{1}{2}\sum_{w:uw\in\Delta}\lambda_{uw}$ for each $u\in [n]$. The partial derivative of $\lambda$ w.r.t.~$x_v$ for $v\in [n]$, is then the polynomial $\sum_{u\in[n]}\alpha_u x_u$, where 
\[
\alpha_u=\left\{\begin{array}{cl} 
\lambda_{uv} & \mbox{ if $uv$ is an edge of $\Delta$}\\
-\sum_{w: uw\in\Delta} \lambda_{uw}  & \mbox{ if $u=v$}\\
0 & \mbox{ otherwise.} \end{array} \right.
\]
Equation \eqref{eq: Stress_i to Stress_j} asserts that $(\alpha_u)_{u\in [n]}$ is an affine dependence of the point configuration $(p(u))_{u\in [n]}$, and moreover, that if $\Delta$ 
is not a simplex and has no missing facets,
then {\em every} affine dependence of $(p(u))_{u\in [n]}$ is a linear combination of dependencies of the above form, possibly coming from different $2$-stresses and from partial derivatives w.r.t~different variables.\footnote{In fact, it follows from \cite[Theorem 1.14]{GHT} along with \cite[Theorem 1.2]{CJT} that if $p$ is a generic embedding, then there is a {\em single} $2$-stress $\lambda$ whose partial derivatives span $\Stress_1(\Delta,p)$.}
		
	That the right-hand side of eq.~\eqref{eq: Stress_i to Stress_j} is contained in the left-hand side is immediate from the definition of stress spaces, and so we only need to prove the reverse inclusion.
We first prove the following statement.

	\begin{proposition} \label{prop:socle and missing faces}
	Let $(\Delta, p)$ be either the boundary complex of a simplicial $d$-polytope with its natural embedding $p$, or a simplicial $(d-1)$-sphere with a generic embedding $p$, and let $\Theta$ be the l.s.o.p.~of $R/I_\Delta$ determined by $p$. Then 
	\begin{eqnarray*}
	\dim_\R \soc^R\big(R/\big(I_\Delta+(\Theta,\ell)\big)\big)_k &=& m_{d-k}(\Delta) \quad \mbox{ if } \textstyle k<\lceil \frac{d}{2}\rceil -1.\\
	\dim_\R \soc^R\big(R/\big(I_\Delta+(\Theta,\ell)\big)\big)_k &\geq& m_{d-k}(\Delta) \quad  \mbox{ if } \textstyle k=\lceil \frac{d}{2}\rceil -1.
	\end{eqnarray*}
	\end{proposition}
\begin{proof}
Observe that $R/(\Theta)R$ and $R/(\Theta,\ell)R$ are isomorphic to polynomial rings with $n-d$ and $n-d-1$ variables, respectively.
Since $\soc^R\big(R/\big(I_\Delta+(\Theta,\ell)\big)\big)=\soc^{R/(\Theta,\ell)R}\big(R/\big(I_\Delta+(\Theta,\ell)\big)\big)$, we conclude from
Lemma \ref{lem:basic-prop-of-Betti}(2) that for all $k$,
\begin{align}
\label{3-0}
\dim_\R \soc^R\big(R/\big(I_\Delta+(\Theta,\ell)\big)\big)_k
=\beta_{n-d-1,n-d-1+k}^{R/(\Theta,\ell)R}\big(R/\big(I_\Delta+(\Theta, \ell)\big)\big).\end{align}
Now, by equations \eqref{2-2} and \eqref{2-3},
\begin{align}
\label{3-1}
\textstyle
\beta_{n-d-1,n-d-1+k}^{R/(\Theta,\ell)R}\big(R/\big(I_\Delta+(\Theta,\ell)\big)\big)=
\beta_{n-d-1,n-d-1+k}^{R/(\Theta)R}\big(R/\big(I_\Delta+(\Theta)\big)\big) \ \ \mbox{ for }k <\big\lceil \frac{d}{2}\big\rceil-1
\end{align}
and
\begin{align}
\label{3-2}
\textstyle
\beta_{n-d-1,n-d-1+k}^{R/(\Theta,\ell)R}\big(R/\big(I_\Delta+(\Theta,\ell)\big)\big)\geq
\beta_{n-d-1,n-d-1+k}^{R/(\Theta)R}\big(R/\big(I_\Delta+(\Theta)\big)\big) \ \ \mbox{for }k = \big\lceil \frac{d}{2}\rceil-1.
\end{align}
But
\begin{align*}
\beta_{n-d-1,n-d-1+k}^{R/(\Theta)R}\big(R/\big(I_\Delta+(\Theta)\big)\big)
&=\beta_{n-d-1,n-d-1+k}^{R}(R/I_\Delta) & \mbox{ (by eq.~\eqref{eq:*})}\\
&=\beta_{1,d+1-k}^R(R/I_\Delta) & \mbox{ (by Lemma \ref{lem:symmetry})}\\
&=m_{d-k}(\Delta) & \mbox{ (by eq.~\eqref{2-1})}.
\end{align*}
The last equation together with equations  \eqref{3-0}, \eqref{3-1} and \eqref{3-2} completes the proof.
\end{proof}

We are ready to prove the main result of this section.

\smallskip\noindent {\it Proof of Theorem \ref{thm: Stress_i to Stress_j}: \ }
Let $J=I_\Delta+(\Theta(p),\ell)$, where $\Theta(p)$ is the l.s.o.p.~determined by $p$.
Consider the natural isomorphism $\Stress(\Delta,p) \cong J^\perp$ explained in Section 2.3.
Keeping the $R$-module structure of $J^\perp$ in mind, equation \eqref{eq: Stress_i to Stress_j} is equivalent to the claim that as an $R$-module, $J^\perp$ has no generators of degree $>-i$. Thus, to prove the theorem, it suffices to show that for $k<i$, $\mu_{-k}(J^\perp)=0$.

Since $J^\perp$ is isomorphic to $(R/J)^\vee$, it follows from \eqref{2-0} that for all $k$, $\mu_{-k}(J^\perp)=r_{k}(R/J)$. Now, by Proposition \ref{prop:socle and missing faces}, if $k< \lceil \frac{d}{2}\rceil -1$ then $r_k(R/J) = m_{d-k}(\Delta)$. Hence for $k <i$, 
\[
\mu_{-k}(J^\perp) = r_k(R/J)=m_{d-k}(\Delta)=0,
\]
where the last step follows from the assumption that $\Delta$ has no missing faces of dimension $\geq d-i+1$. This completes the proof.
\endproof

Proposition \ref{prop:socle and missing faces} played a crucial role in this section and will be used again in the next section. Thus, it is helpful to illustrate this result on several examples. Let $\{e_1, \ldots, e_d\}$ be the standard basis of $\R^d$.

\begin{example}
	Let $\Sigma$ be the boundary complex of the $d$-simplex with $p(i)=e_i$ for $1\leq i\leq d$, and $p(d+1)=-\sum_{i=1}^d e_i$.  Then $m_d(\Sigma)=1$, $m_j(\Sigma)=0$ for $j<d$, and  $$R/(I_\Sigma+(\Theta,\ell))=\R[x_1,\dots x_{d+1}]/(x_1x_2\dots x_{d+1}, x_1-x_{d+1}, x_2-x_{d+1}, \dots, x_d-x_{d+1}, x_1+\dots+x_{d+1})=\R.$$ Hence, the socle is concentrated in degree zero and has dimension one.  This agrees with Proposition~\ref{prop:socle and missing faces} asserting that 
	$$\dim_\R \soc^R (R/(I_\Sigma+(\Theta,\ell)))_k =\begin{cases}
		1=m_{d-0}(\Sigma) & \mbox{ if } k=0 \\
		0=m_{d-k}(\Sigma) & \mbox{ if } 1\leq k< \lceil \frac{d}{2}\rceil -1\\
		0 \geq m_{d-k}(\Sigma) & \mbox{ if } k= \lceil \frac{d}{2}\rceil -1
	\end{cases}.$$ In particular, the inequality in the case of $k=\lceil\frac{d}{2}\rceil-1$ holds as equality. 
\end{example}

\begin{example}
To understand Proposition \ref{prop:socle and missing faces}, it is useful to consider {\em Betti tables}. The $(j,i)$-th entry of such a table is $\beta_{i,i+j}$ (i.e., $i$ indicates the column number and $j$ indicates the row number). 

Let $P\subset \R^5$ be the convex hull of $\pm e_1,\dots,\pm e_5,\frac 1 4 \sum_{k=1}^5 e_i$, and let $\Gamma=\partial P$ be the boundary complex of $P$ with its natural embedding $p$.
Note that $P$ is obtained from the $5$-dimensional cross-polytope by a stellar subdivision at a facet.
The left table below is the Betti table of $A_\Gamma=R/I_\Gamma$ over $R$ 
and the right table is that of $\overline A_\Gamma=R/(I_\Gamma+(\Theta,\ell))$ over $\overline R=R/(\Theta,\ell)R$, where $\Theta=\Theta(p)$. Both tables were computed using the computer algebra system Macaulay2 \cite{M2}.
\begin{center}
\small
\begin{tabular}{c|ccccccc}
& 0 & 1 & 2 & 3 & 4 & 5 & 6\\
\hline
0 & 1 & \\
1 &  & 10 & 15 & 10 & 5 & \bf {\color{red}1} & \\
2 &  & 0 & 10 & 10 & 0 & \bf {\color{blue} 0} & \\
3 &  & \bf {\color{blue}0} & 0 & 10 &10 & 0 &\\
4 &  & \bf {\color{red}1} & 5 & 10 & 15 & 10 & \\
5 & & & && & & 1
\end{tabular}
\hspace{20pt}
\begin{tabular}{c|ccccccc}
& 0 & 1 & 2 & 3 & 4 & 5\\
\hline
0 & 1 & &  & & &  \\
1 &  & 10 & 15 & 10 & 5 & \bf \color{red}1 \\
2 &   & 0 & 15 &31 & 21 & \bf \color{blue} 5
\end{tabular}
\end{center}

For instance, the left table says that $\beta_{0,0}^R(A_\Gamma)=1$,
$\beta_{1,2}^R(A_\Gamma)=10$,
$\beta_{1,3}^R(A_\Gamma)=0$,
$\beta_{1,4}^R(A_\Gamma)=0$,
$\beta_{1,5}^R(A_\Gamma)=1$, and so on.
Note that the left table is symmetric as it should be by Lemma \ref{lem:symmetry}.
The first column of the left table consists of numbers $m_k(\Gamma)=\beta_{1,k+1}^R(A_\Gamma)$ (see \eqref{2-1}), and so
\[m_1(\Gamma)=10,\ m_2(\Gamma)=0,\ m_3(\Gamma)=0, \mbox{ and }m_4(\Gamma)=1.\]
The 5th column of the right table consists of numbers
$\dim_\R \soc^R(\overline A_\Gamma)_k
=\beta^{\overline R}_{5,5+k}(\overline A_\Gamma)$
(see Lemma \ref{lem:basic-prop-of-Betti}(2)). Thus, we conclude that
\[
\dim_\R \soc^R (\overline A_\Gamma)_1=1
\mbox{ and }
\dim_\R \soc^R (\overline A_\Gamma)_2=5.
\]

Lemmas \ref{lem:nonzerodiv} and \ref{lem:MN} imply that the first row of the left table coincides with the first row of the right table, while each entry of the second row of the right table is at least as large as the corresponding entry of the left table. In particular, keeping in mind the symmetry of the left table, we obtain that
\begin{align}
\label{exeq1}
1=
\dim_\R \soc^R (R/(I_\Gamma+(\Theta,\ell)))_1
=\beta_{5,6}^{\overline R} (\overline A_\Gamma)=\beta_{5,6}^R(A_\Gamma)=\beta_{1,5}^R(A_\Gamma)=\mu_4(\Gamma)=1
\end{align}
(see the entries marked in red)
and
\begin{align}
\label{exeq2}
5=
\dim_\R \soc^R (R/(I_\Gamma+(\Theta,\ell)))_2
=\beta_{5,7}^{\overline R}(\overline A_\Gamma)
\geq
\beta_{5,7}^R(A_\Gamma)=\beta_{1,4}^R(A_\Gamma)=\mu_3(\Gamma)=0
\end{align}
(see the entries marked in blue).
Equations \eqref{exeq1} and \eqref{exeq2} illustrate the proof and the statement of Proposition \ref{prop:socle and missing faces}.
Equation \eqref{exeq2} also shows that the inequity in the proposition could be strict.
\end{example}

\begin{example}
Let $\Omega$ be the boundary complex of the $4$-dimensional cross-polytope with vertices $\pm e_1,\dots,\pm e_4$ in $\R^4$.
The Betti tables of $R/I_\Omega$ and $R/(I_{\Omega}+(\Theta,\ell))$ (over $R$ and $\overline{R}=R/(\Theta,\ell)R$, respectively)
are as follows.
\begin{center}
\small
\begin{tabular}{c|ccccccc}
& 0 & 1 & 2 & 3 & 4 \\
\hline
0 & 1 & \\
1 &  & 4 & 0 & \bf \color{red} 0 &  \\
2 &  & 0 & 6 & 0 & \\
3 &  & \bf \color{red}0 & 0 &4  &\\
4 & & & && 1
\end{tabular}
\hspace{20pt}
\begin{tabular}{c|ccccccc}
& 0 & 1 & 2 & 3 \\
\hline
0 & 1 & &  & & &  \\
1 &  &4 & 2 & \bf \color{red} 0  \\
2 &   & 0 & 3 &2 
\end{tabular}
\end{center}

Since $d=4$,
Lemmas \ref{lem:nonzerodiv} and \ref{lem:MN} only imply that the entries of the first row of the right table are at least as large as the corresponding entries of the left table. Consequently,
\[
\dim_\R \soc^R(R/(I_\Omega+(\Theta,\ell)))_1
=
\beta_{3,4}^{\overline{R}}(R/(I_\Omega+(\Theta,\ell)))
\geq 
\beta_{3,4}^R(R/I_\Omega)=
\beta_{1,4}^R(R/I_\Omega)=m_3(\Omega),
\]
which is the inequality guaranteed by Proposition \ref{prop:socle and missing faces}.
In this example, the values of the left- and right-hand sides are both zeros (see the entries marked in red); hence they are equal.
\end{example}

\begin{remark}
	When $d=2k+1$ is an odd number and $\Delta$ is a simplicial $2k$-sphere with a generic embedding, $\dim_\R \soc^R (R/(I_\Gamma+(\Theta,\ell)))_k=\dim_\R R/(I_\Gamma+(\Theta,\ell))_k =g_k(\Delta)$. This number is often strictly larger than $m_{k+1}$ (see \eqref{exeq2}).  Corollary \ref{cor:k-stacked}, which we prove in the following section, discusses certain sufficient conditions for the equality to hold.
	
	When $\Delta$ is a simplicial $3$-sphere with $m_3(\Delta)=m_4(\Delta)=0$ and $p$ is generic, the results of \cite{CJT} imply that \eqref{eq: Stress_i to Stress_j} continues to hold for $i=\frac{4}{2}=2$ and $k=i-1=1$. Consequently, in this case,  $\dim_\R \soc^R (R/(I_\Delta+(\Theta,\ell)))_1=0=m_3(\Delta)$, and so the inequality of Proposition \ref{prop:socle and missing faces} holds as equality. For $d=2i\geq 6$ and $k=i-1$, there do exist examples where the inequality of Proposition \ref{prop:socle and missing faces} is strict even if one assumes that $m_{i+1}=\cdots=m_d=0$. Such examples will be discussed in a future paper.  As a consequence,  the statement of  Theorem \ref{thm: Stress_i to Stress_j} does not extend to the case of $d=2i\geq 6$, and hence at present the case $d=2i$ of Conjecture \ref{Old Conjecture 1.2} remains open.
\end{remark}

\section{$k$-stacked polytopes and a strengthening of the $g$-theorem}
The goal of this short section is to use results of Section 3, most notably Proposition \ref{prop:socle and missing faces}, to prove Corollaries \ref{cor:inequalities on g_k} and \ref{cor:k-stacked}. Throughout this section, we assume that $\Delta$ is a simplicial $(d-1)$-sphere with a generic $d$-embedding $p$ and we let $\Theta$ be the l.s.o.p.~of $R/I_\Delta$ determined by $p$. 

\smallskip\noindent {\it Proof of Corollary \ref{cor:inequalities on g_k}: \ }
Let $M=R/(I_\Delta+(\Theta,\ell))$ and let $k\leq\lceil \frac{d}{2}\rceil-1$. Then $M_k \supseteq \soc^R(M)_k$, and by the $g$-theorem, $\dim_\R M_k=g_k(\Delta)$. The inequality $g_k(\Delta)\geq m_{d-k}(\Delta)$ then follows from Proposition \ref{prop:socle and missing faces}.
Now, by definition of socles, $\soc^R(M)_k$ is an ideal of $M$. Furthermore, by Proposition \ref{prop:socle and missing faces}, the Hilbert function of the quotient algebra $A:=M/\big(\soc^R(M)_k)$ satisfies the following:
\begin{eqnarray*}
&&\dim_\R A_i = g_i(\Delta)  \qquad \qquad \qquad \mbox{ if $0\leq i \leq \lceil \frac d 2 \rceil, \ i\neq k$;}\\ 
&&\dim_\R A_i \leq g_k(\Delta)-m_{d-k}(\Delta) \;\;\, \mbox{ if $i=k$.}
\end{eqnarray*}
The inequality $0\leq g_{k+1}(\Delta)\leq \big(g_k(\Delta)-m_{d-k}(\Delta)\big)^{\langle k \rangle}$ then follows from Macaulay's theorem.
\endproof

\begin{remark} The join of two copies of the boundary complex of a $k$-simplex is a simplicial $(2k-1)$-sphere with exactly two missing $k$-faces and $g_k=1$. Hence the inequality $g_k\geq m_{d-k}$ of Corollary \ref{cor:inequalities on g_k} does not hold in the case of $k=\frac d 2$.
\end{remark}

\smallskip\noindent {\it Proof of Corollary \ref{cor:k-stacked}: \ }
Let $1\leq k \leq \lfloor \frac{d}{2}\rfloor-1$ and assume that $g_k(\Delta)=m_{d-k}(\Delta)$. Then by Corollary \ref{cor:inequalities on g_k}, $0\leq g_{k+1}\leq 0^{\langle k\rangle}$. Thus, $g_{k+1}(\Delta)=0$, and so by a result of Murai and Nevo \cite{MuraiNevo2013}, $\Delta$ is $k$-stacked. 

Conversely, assume that $1\leq k\leq \lceil \frac{d}{2}\rceil-1$ and that $\Delta$ is $k$-stacked. Then there exists a homology $d$-ball $\Gamma$ such that $\partial\Gamma=\Delta$ and  $f_{i-1}(\Gamma,\partial\Gamma)=0$ for all $i\leq d-k$. In fact, such $\Gamma$ is unique: the faces of 
$(\Gamma,\partial\Gamma)$ are the sets $F\subseteq V(\Delta)$ having the property that $F$ is not a face of $\Delta$ but all $(d-k)$-subsets of $F$ are faces of $\Delta$
(see \cite[Theorem 2.3]{MuraiNevo2013} and references provided there). Therefore, the $(d-k)$-faces of $(\Gamma,\partial \Gamma)$ are precisely the missing $(d-k)$-faces of $\Delta$, that is, $f_{d-k}(\Gamma,\partial\Gamma)=m_{d-k}(\Delta)$. We then have
\[g_k(\Delta)=h_k(\Gamma)=h_{d+1-k}(\Gamma,\partial\Gamma)=f_{d-k}(\Gamma,\partial\Gamma)+\sum_{i=1}^{d-k}(-1)^{j-i}\binom{d+1-i}{k}f_{i-1}(\Gamma, \partial \Gamma)=m_{d-k}(\Delta).\]
Here, the first and the third steps follow from the definitions of the $g$- and $h$-numbers; for the second step, see \cite[II Section 7]{Stanley96}. The result follows.
\endproof
	
\section{Every edge participates in a stress}
A simplicial polytope $P$ is called {\em prime} if $P$ has no missing facets.
The goal of this section is to prove Theorem \ref{thm:edges in stresses} asserting that if $d\geq 4$ and $P$ is a prime simplicial $d$-polytope that is not a simplex, then every edge of $P$ participates in an affine $2$-stress. In other words, if $G(P)=\skel_1(\partial P)$ is the graph of $P$ with its natural embedding $p$ and $e$ is an edge of $G(P)$, then there exists $\lambda\in\Stress_2(G(P), p)$ such that $\lambda_e\neq 0$. Our proof relies on several results from the classical rigidity theory of frameworks, and we refer the reader to \cite{AsimowRothI,AsimowRothII} (see also \cite[Chapter 6]{Lee-notes}) for background and all undefined terminology.

Let $G=(V,E)$ be a finite graph with an embedding $p:V\to \R^d$ and assume that the affine span of $p(V)$ is the entire $\R^d$. The pair $(G,p)$ is called a \emph{$d$-framework}. It is known and not hard to prove that for every $d$-framework $(G,p)$, $\dim \Stress_2(G,p)\geq f_1(G)-df_0(G)+\binom{d+1}{2}$. A $d$-framework $(G,p)$ is called {\em infinitesimally rigid} if $\dim \Stress_2(G,p)= f_1(G)-df_0(G)+\binom{d+1}{2}$. It is an easy consequence of this discussion that if $(G,p)$ is infinitesimally rigid and $e$ is a missing edge of $G(P)$, then adding $e$ to $G$ preserves infinitesimal rigidity and increases the dimension of the space of affine $2$-stresses by exactly one. The following result asserts that the graph of any simplicial polytope is infinitesimally rigid.

\begin{theorem} \label{thm:rigidity of polytopes}
    	Let $d\geq 3$ and let $P\subset \R^d$ be a simplicial $d$-polytope. Then the graph of $P$ with its natural embedding $p$ is infinitesimally rigid. 
 \end{theorem}
	
	Since $g_2(P)=f_1(P)-df_0(P)+\binom{d+1}{2}$, Theorem \ref{thm:rigidity of polytopes} is a special case of Theorem \ref{g-thm-stresses} for polytopes. It is also a special case of a result of Whiteley \cite{Whiteley-84}. In the case of $d=3$, it was proved by Dehn, and is often referred to as Dehn's lemma. Since any simplicial $3$-polytope has $g_2=0$, Dehn's lemma asserts that the graph of a simplicial $3$-polytope $P$ supports no non-trivial affine $2$-stresses.  
	Therefore, if $e$ is a missing edge of $G(P)$, then there is a unique $\lambda\in\Stress_2(G(P)\cup e, p)$ such that $\lambda_e=1$. 
	
	A strengthening of Dehn's lemma was proved in \cite[Lemma 5.1]{NZ-reconstruction}. Lemma \ref{stress on 3-polytope union an edge} provides a further strengthening. A big portion of our proof is based on \cite[Theorem 6.17]{Lee-notes}. 
	 To state the result, we introduce some notation.
	Let $P$ be a simplicial $d$-polytope. If $\tau$ is a missing facet of $\partial P$, then cutting $P$ along the hyperplane determined by the vertices of $\tau$ decomposes $P$ into two polytopes $P_1$ and $P_2$; we write $P=P_1\#P_2$.  Induction on the number of missing facets then implies that $P$ has a unique decomposition $P_1\#P_2\#\cdots\#P_m$, where each $P_i$ is a prime $d$-polytope. This decomposition is called the {\em prime decomposition} of $P$, and each $P_i$ is called a {\em prime component} of~$P$.

\begin{lemma}\label{stress on 3-polytope union an edge}
		Let $\partial P$ be the boundary complex of a simplicial $3$-polytope $P$ with its natural embedding $p$ and let $e=ab$ be a missing edge of $\partial P$. Let $P_1, \dots, P_k$ be the prime components of $P$ that have non-empty intersections with the interior of the line segment $[p(a), p(b)]$. Let $Q=P_1\#P_2\#\cdots \#P_k\subset P$ and let $\lambda$ be the unique affine $2$-stress on $(G(Q)\cup e, p)$ such that $\lambda_e=1$. Then $\lambda_{e'}< 0$ for every edge $e'$ of $G(Q)$ incident with either $a$ or $b$.
	\end{lemma} 
	\proof
By definition of $Q$, there is only one prime component of $Q$ that has $p(a)$ as its vertex. This means that $a$ does not belong to any missing $2$-face of $\partial Q$, and so $\lk_{\partial Q}(a)$ is an {\em induced} cycle in $\partial Q$. Also, since $\lambda_e=1$, the balancing condition at $a$ (see eq.~\eqref{eq:balancing}) guarantees that at least some of the edges incident with $a$ are assigned negative weights by $\lambda$.
	
	Label each edge of $Q$ with $+,-,0$ depending on whether its $\lambda$-weight  is $>0, <0$ or $0$. Delete every vertex of $Q$ that is adjacent only to edges labeled with $0$. The convex hull of the remaining vertices is $3$-dimensional, but it may have non-triangular $2$-faces; we arbitrarily triangulate all such faces (by adding diagonals) and label all new edges with $0$. The boundary of the resulting object is a simplicial $2$-sphere which we denote by $S$, and $\lambda$ (with added zeros) induces an affine $2$-stress $\lambda'$ on $(S\cup e, p)$. By construction, every vertex of $S$ participates in $\lambda'$; furthermore, every edge with a non-zero $\lambda'$-weight is an edge of $G(Q)\cup e$.
		
Let $F$ be a $2$-face of $S$. In each corner $v$ of $F$, place the label $0$ if the two edges containing $v$ are of the same sign, $1$ if they are of opposite sign, and $1/2$ if one is zero and the other one is non-zero. Let $n_F$ be the sum of corner labels around $F$. Also, for each vertex $v$ of $S$, let $m_v$ be the sum of corner labels at $v$. Then $n_F\leq 2$ for every $2$-face $F$ of $S$, and $m_v\geq 4$ for every vertex $v\neq a, b$ of $S$ (this is an easy consequence of the balancing condition at $v$; see \cite[Claim 1 in the proof of Theorem 6.17]{Lee-notes}). Hence
		$$4(f_0(S)-2)\leq \sum_{v\in V(S)} m_v=\sum_F n_F\leq 2(2f_0(S)-4).$$
		Thus all inequalities in this equation must hold as equalities. We conclude that $m_v=4$ for all $v\in V(S)\backslash\{a,b\}$, $n_F=2$ for all $2$-faces $F$ of $S$, and $m_a=m_b=0$.

			Since $m_a=0$ and since some edges incident with $a$ have negative signs, it follows that $\lambda'_{e'}<0$ for {\em all} edges $e'$ of $S$ incident with $a$. (If at least one edge were labeled with $+$ or $0$, the value of $m_a$ would be at least $1/2$.) Thus, every edge of $S$ incident with $a$ is also an edge of $\partial Q$.
			
			Let $G=\{a,x,y\}$ be a $2$-face of $S$. Since $n_{G}=2$ and $\lambda'_{ax}, \lambda'_{ay}<0$, we obtain that $\lambda'_{xy}>0$. Consequently, $xy$ is also an edge of $\partial Q$. In other words, all vertices of $\lk_S(a)$ are vertices of $\lk_{\partial Q}(a)$ and all edges of $\lk_S(a)$ are edges of $\partial Q$. Since $\lk_{\partial Q}(a)$ is an induced cycle in $\partial Q$, $\lk_S(a)$ must be a subcomplex of $\lk_{\partial Q}(a)$. Finally, since both of them are cycles, we must have that $\lk_S(a)=\lk_{\partial Q}(a)$. Thus, the set of edges of $\partial Q$ incident with $a$ coincides with the set of edges of $S$ incident with $a$. Our result follows since all edges of $S$ incident with $a$ have negative signs. A similar argument applies to $b$.
			\endproof
			
	The proof of Theorem \ref{thm:edges in stresses} relies on Lemma \ref{stress on 3-polytope union an edge} and on two results from  rigidity theory, known as the gluing and cone lemmas (see, for instance, 	\cite[Theorem 2]{AsimowRothII} and \cite[Corollary 6.12]{Lee-notes} for the gluing lemma, and \cite[Claim 1 of Theorem 6.19]{Lee-notes} and also \cite[Lemma 3.2]{NZ-reconstruction} for the cone lemma). 
			
\begin{lemma}[The Gluing Lemma]
    	Let $(G,p)$ be a $d$-framework, and let $G'$ and $G''$ be subgraphs of $G$. If $(G',p)$ and $(G'',p)$ are infinitesimally rigid and have $d$ affinely independent vertices in common (i.e., the framework $(G'\cap G'', p)$ affinely spans a subspace of dimension at least $d-1$), then the union $(G'\cup G'',p)$ is also infinitesimally rigid.
    \end{lemma}		
			
	Let $G=(V,E)$ be a graph, and let $u\notin V$ be a new vertex. The {\em cone} over $G$ is the graph $u*G=(V\cup u, E\cup\{uv : v\in V\})$. For instance, if $\Delta$ is a simplicial complex and $u$ is a vertex of $\Delta$, then the graph of $\st_\Delta(u)$ is the cone over the graph of $\lk_\Delta(u)$.
		
	\begin{lemma}[The Cone Lemma] 
	Let $(u*G,p)$ be a framework in $\R^d$, let $H$ be a hyperplane not containing $p(u)$, and let $\pi$ be the central projection from $p(u)$ onto $H$. Assume also that for all $v\in V(G)$, $p(u)-p(v)$ is neither zero nor parallel to $H$. Then $(u*G,p)$ is infinitesimally rigid in $\R^d$ if and only if $(G,\pi\circ p)$ is infinitesimally rigid in $H\cong \R^{d-1}$. Furthermore, there is an isomorphism $\varphi: \Stress_2(G, \pi\circ p) \to \Stress_2(u*G,p)$ with the following property: for all $\lambda\in \Stress_2(G, \pi\circ p)$ and $e\in E(G)$,  $\lambda_e\neq 0$ if and only if $(\varphi(\lambda))_e\neq 0$. 
	\end{lemma}
	
	One immediate but very useful consequence of Theorem \ref{thm:rigidity of polytopes} and the Cone Lemma is that stars of faces of simplicial polytopes are infinitesimally rigid. In more detail, if $d\geq 4$, $P$ is a simplicial $d$-polytope, and $\tau\in\partial P$ is a non-empty face of dimension $\leq d-4$, then the framework $(G(\st_{\partial P}(\tau)),p)$ is infinitesimally rigid. 
	
	We are now ready to prove Theorem \ref{thm:edges in stresses}. The proof is by induction on $d$. The following lemma takes care of the base case.
	
	\begin{lemma}\label{lm: 4-dimensional case}
		Let $\partial P$ be the boundary complex of a prime simplicial $4$-polytope $P$ with its natural embedding $p$. If $P$ is not a simplex, then every edge of $P$ participates in an affine $2$-stress on $(G(P), p)$. 
	\end{lemma}
	\begin{proof}
		Let $e=ab$ be an edge of $\partial P$, and consider the link of $e$ in $\partial P$. Assume first that $\lk_{\partial P}(e)$ is an induced cycle of $G(P)$. Then $\lk_{\partial P}(e)$ is an induced subcomplex of $\partial P$ unless it is a cycle of length three and its vertices, which we denote by $x,y,z$, form a $2$-face of $\partial P$. If this were the case, then either one of $\{a,x,y,z\}$, $\{b,x,y,z\}$ would be a missing $3$-face of $\partial P$ or $P$ would be a simplex, neither of which is possible by the assumptions of the theorem. Thus $\lk_{\partial P}(e)$ is an induced subcomplex of $\partial P$, and hence so is the star $\st_{\partial P}(e)$. As the star is contractible, it follows that $W:=V(P)\backslash V(\st_{\partial P}(e))$ is nonempty; furthermore, by the Alexander duality, the subcomplex of $\partial P$ induced by $W$ is acyclic, and so the subgraph of $G(P)$ induced by $W$ is connected. If $u,w \in W$ are adjacent vertices, then $(\st_{\partial P}(u), p)$ and $(\st_{\partial P}(w),p)$ share at least four affinely independent vertices from $(\st_{\partial P} (uw),p)$. The Gluing Lemma then implies that the framework $\big(\cup_{v\in W} G(\st_{\partial P}(v)), p\big )$ is infinitesimally rigid. It contains $a$ and $b$, but not the edge $e=ab$.
	Thus, adding $e$ to this framework creates a new affine $2$-stress with a non-zero weight on $e$.
		
		Assume now that $\lk_{\partial P}(e)$ is not an induced cycle of $\partial P$, and let $cd \in \partial P$ be a missing edge of $\lk_{\partial P}(e)$. Since $\{a,b,c,d\}$ is not a missing facet of $\partial P$ and since $\{a,b,c\}, \{a,b,d\}\in \partial P$, at least one of $\{a,c,d\}$ and $\{b,c,d\}$, say $\{a,c,d\}$, must be a missing 2-face. Then $ab \in \lk_{\partial P}(d)$, $ac$ is a missing edge of $\lk_{\partial P}(d)$, and $bc$ is either an edge or a missing edge. Consider the case where both $ac$ and $bc$ are missing edges. Let $Q=P/d$ be the quotient polytope with its natural embedding $q$. Let $Q'=Q_1\#\cdots \# Q_k\subset Q$, where $Q_1, \dots, Q_k$ are the prime components of $Q$ that have non-empty intersections with the interior of $[q(a),q(c)]$ or  $[q(b),q(c)]$. Note that the relative interior of $[q(a),q(b),q(c)]$ lies in the interior of $Q'$ while $[q(a), q(b)]$ lies on the boundary of $Q'$. This means that $Q_1 \cap [q(a),q(b),q(c)], Q_2\cap [q(a),q(b),q(c)], \dots, Q_k \cap [q(a),q(b),q(c)]$ decompose the triangle $[q(a),q(b),q(c)]$ into polygons; furthermore, the vertices of these polygons consist of $q(a)$, $q(b)$, $q(c)$, and some interior points of $[q(a),q(c)]$ and $[q(b),q(c)]$.  In particular,  there is a polygon containing $[q(a),q(b)]$. This polygon must contain a nonzero (in length) part of $[q(a),q(c)]$ or of $[q(b),q(c)]$. Without loss of generality, assume it contains $[q(a),q(c')]$ with $q(c')\neq q(a)$. Then by Lemma \ref{stress on 3-polytope union an edge}, there is an affine stress on $(G(Q)\cup ac, q)$ with a non-zero weight on $ab$. By the Cone Lemma, this stress corresponds to an affine $2$-stress on $(G(P), p)$ with a non-zero weight on $e=ab$, proving the result. The treatment of the case where $bc$ is an edge of $\lk_{\partial P}(d)$ is very similar.
	\end{proof}

We are now in a position to prove Theorem \ref{thm:edges in stresses}. In fact, we verify the following stronger result.
\begin{theorem} Let $d\geq 4$. Let $\partial P$ be the boundary complex of a simplicial $d$-polytope $P$ with its natural embedding $p$. Let $e$ be an edge of $\partial P$.
	\begin{enumerate}[(1)]
		\item If $P$ is prime and $P$ is not a simplex, then there is an affine 2-stress on $(G(P), p)$ with a non-zero weight on $e$.
		\item If $e$ is not a part of any missing $(d-1)$-face of $\partial P$ and $\lk_{\partial P}(e)$ is not the boundary of a $(d-2)$-simplex, then there is an affine 2-stress on $(G(P), p)$ with a non-zero weight on $e$.
	\end{enumerate}
\end{theorem}
\begin{proof}
	We prove both statements simultaneously by induction on $d$. The base case $d=4$ of statement (1) is verified by Lemma \ref{lm: 4-dimensional case}. Assume that (1) holds for some $d\geq 4$. Our strategy is to prove that then (2) holds for the same $d$, and then use this to show that (1) holds for $d+1$.
	
	To prove that (2) holds for $d$, let $P=P_1\#\cdots \#P_k$ be the prime decomposition of $P$. If $e$ is not a part of a missing facet of $\partial P$, then there is a unique prime component $P_i$ that contains $e$, and so $\lk_{\partial P}(e)=\lk_{\partial P_i}(e)$. That $\lk_{\partial P}(e)$ is not the boundary of a $(d-2)$-simplex implies that $P_i$ is a prime $d$-polytope that is not a simplex. Hence by (1) (in dimension $d$), a desired stress does exist.
	
	To prove that (1) holds for $d+1$, assume $P$ is a prime simplicial $(d+1)$-polytope that is not a simplex, and let $e$ be an edge of $P$. If $e$ is a part of a missing $(d-1)$-face $F$, then $e$ is a missing edge of  the $2$-sphere $\lk_{\partial P}(F \backslash e)$. By the Cone Lemma,  $(G(\st_{\partial P}(F \backslash e)) \cup e, p)$ is infinitesimally rigid and there is an affine $2$-stress $\lambda$ with $\lambda_e\neq 0$. We are done in this case. 
	
	Now assume that $e$ is not a part of any missing $(d-1)$-face of $\partial P$. Consider the quotient polytope $Q=P/e$. There are two possible cases:  1) $Q$ is not a simplex, and 2) $Q$ is a simplex. In the former case, there exists a vertex $v$ of $\partial Q=\lk_{\partial P}(e)$ such that $\lk_{\partial Q}(v)=\lk_{\partial P}(e \cup v)$ is not the boundary complex of a $(d-2)$-simplex. In other words, $\lk_{\partial(P/v)}(e)$ is not the boundary complex of a $(d-2)$-simplex. Furthermore, $e$ is not a part of a missing $(d-1)$-face of $\partial (P/v)$. 
	(This is because any missing $(d-1)$-face $F$ of $\partial (P/v)$ must be a missing face of $\partial P$ or else $v\cup F$ would be a missing facet of $\partial P$. Hence if $e$ were a part of a missing $(d-1)$-face $F$ of $\partial (P/v)$, then $e$ would be a part of a missing $(d-1)$-face of $P$, contradicting our assumptions on $e$.)
	Thus by statement (2) (in dimension $d$) applied to $T=P/v$ with its natural embedding $t$, there is an affine stress $\lambda$ on $\big(G(\lk_{\partial P}(v)), t\big)$ such that $\lambda_e\neq 0$. By the Cone Lemma, this stress corresponds to a stress on $(G(\st_{\partial P}(v)), p)$ with a non-zero weight on $e$, which completes the proof in this case.
	
	In the case that $Q=P/e$ is a simplex, the same argument as in the first part of the proof of Lemma \ref{lm: 4-dimensional case} shows that both $\lk_{\partial P}(e)$ and $\st_{\partial P}(e)$ are induced subcomplexes of $\partial P$ and that $e$ participates in some stress on $(G(P), p)$, thus completing the proof in this case.
\end{proof}

Theorem \ref{thm:edges in stresses} and the main result of \cite{Z-rigidity} prove the $i=2$ case of the following conjecture that was originally proposed in \cite[Conjecture 3.3]{NZ-AffReconstr}. All other cases of this conjecture remain wide open.

\begin{conjecture}
Let $2\leq i\leq d/2$ and let $\Delta$ be either the boundary complex of a simplicial $d$-polytope with its natural embedding $p$ or a simplicial $(d-1)$-sphere with a generic embedding $p$. Assume further that $\Delta$ has no missing faces of dimension $\geq d-i+1$. Then every $(i-1)$-face of $\Delta$ participates in some affine $i$-stress on $(\Delta,p)$. In particular, the space of affine $i$-stresses of $(\Delta, p)$ determines the $(i-1)$-skeleton of $\Delta$.
\end{conjecture}

Algebraically, this conjecture posits that, under the assumptions of the conjecture, the set of squarefree monomials contained in $(I_\Delta)_i$ is equal to the set of squarefree monomials contained in $\big(I_\Delta+(\Theta(p),\ell)\big)_i$.

\section*{Acknowledgments}  We are grateful to the referee for bringing to our attention reference \cite{GHT} and for several suggestions that improved the presentation.
{\small
	\bibliography{refs}
	\bibliographystyle{plain}
}
	\end{document}